\documentclass{article}
\usepackage{amsmath}
\usepackage{amsthm}
\usepackage[utf8]{inputenc}
\usepackage[T1]{fontenc}
\usepackage[english]{babel}
\usepackage{amssymb}
\usepackage{cleveref}
\usepackage{enumerate}

\usepackage{tikz}
\usetikzlibrary{calc}

\input{mesh.tex}

\newtheorem*{remark}{Remark}

\newtheorem{theorem}{Theorem}
\newtheorem{lemma}{Lemma}
\newtheorem{proposition}{Proposition}
\newtheorem{corollary}{Corollary}

\begin{document}

\title{A Formal Analogue of Euler’s Formula\\for Infinite Planar Regular Graphs}

\author{
	Piotr Jędrzejewicz$^{1}$ and Mikołaj Marciniak$^{2}$ \\
	{\small $^{1}$Nicolaus Copernicus University in Toruń}, 
	{\small $^{2}$University of Szczecin}
}

\maketitle

\begin{abstract}
We present a formal version of the numbers of vertices, edges, and faces for infinite planar regular triangular meshes of degree $r>6$. These numbers are defined via Euler summation of sequences obtained from iterated expansions of a convex combinatorial disk. We prove that these formal quantities satisfy the classical Euler formula, providing 
a combinatorial analogue of Euler's formula for infinite planar graphs. 
\end{abstract}

\newpage

\section*{Introduction}

\subsection*{Platonic solid}
A Platonic solid \cite{Tap21, Cor04} is a convex polyhedron in which each face 
is a regular polygon and each vertex has a fixed degree $r$. 
Each Platonic solid can be drawn as a graph on a plane \cite{Joh66}. 
In the special case where all faces are triangles, 
the number of vertices $v$, edges $e$, and faces $f$ is given by the formulas 
$$v=\frac{12}{6-r}, \; e=\frac{6r}{6-r}, \; f=\frac{4r}{6-r},$$
where $r \in \{3,4,5\}$ \cite{Ric08}.
These formulas can be obtained from elementary properties for graphs in the plane.

A natural extension of this type of graph occurs for $r>=6$. For $r=6$, we obtain an infinite grid which combinatorially corresponds to the hexagonal tiling. For $r>6$, it is necessary to draw a graph 
in the hyperbolic plane or to abandon the condition that all faces are equilateral triangles \cite{Joh66}. 

Furthermore, all such graphs drawn in the plane satisfy the relationship \cite{Ric08} between the number of its vertices, edges, and faces, known as Euler's formula.

\subsection*{Euler summation of infinite series}
The Euler summation is based on assigning to a divergent series 
a finite number calculated as the value in the analytical extension 
of the formal generating function of the series \cite{Att25}. 

Related regularization methods are even useful in physics. 
For example, the equality $$1+2+3+\ldots=-1/12$$ is used in calculating 
the energy of the quantum vacuum field between two plates \cite{Mon25}. 
The formula for the vacuum energy and the resulting Casimir force 
obtained in this way turns out to be correct and can be verified experimentally.

In this paper, we intend to use Euler summation to count the "number" 
of vertices, faces, and edges of infinite regular graphs.
 
\section{Preliminaries}

\subsection{Summation of recursive sequences}
Let $(x_n)_{n=1}^{\infty}$ be a sequence given by the recurrence
\begin{equation*}
	x_{n+2} = \gamma x_{n+1} - x_{n},
\end{equation*}
where $\gamma \neq 2$ is a real number. 
Let $F(t) = \sum_{n=1}^{\infty} x_n t^n$ denote the formal generating function of the series $(x_n)_{n=1}^{\infty}$. 
Multiplying the above recurrence by $t^{n+2}$ and summing over all $n \geq 1$ we obtain
$$\sum\limits_{n=1}^{\infty} x_{n+2} t^{n+2} = \gamma t \sum\limits_{n=1}^{\infty} x_{n+1} t^{n+1} - t^2 \sum\limits_{n=1}^{\infty} x_{n} t^{n}.$$
Next, we express both sides in terms of $F(t)$. Thus
$$F(t) -x_1 t - x_2 t^2 = \gamma t \left(F\left(t\right)- x_1 t\right) - t^2 F(t)$$
and finally after collecting
$$F(t) (1 - \gamma t + t^2) = (t-\gamma t^2) x_1 + t^2 x_2.$$
Dividing both sides by $1 - \gamma t + t^2$ gives
$$ F(t) = \frac{(t-\gamma t^2) x_1 + t^2 x_2}{1-\gamma t + t^2}$$
for all $t$ within the radius of convergence of the formal power series $F(t)$.

The rational function 
$$ F_x(t) = \frac{(t-\gamma t^2) x_1 + t^2 x_2}{1-\gamma t + t^2}$$
is the analytic continuation of the formal generating function $F(t)$ associated with the sequence $(x_j)_{j=1}^{\infty}$. The assumption $\gamma \neq 2$ ensures that the value $F_x (1)$ is well-defined. We define the Euler sum of the (potentially divergent) series $\sum_{n=1}^{\infty} x_n$ by
\begin{equation}
	\label{def:euler_sum}
	\tag{*}
	\sum_{n=1}^{\infty} x_n  := F_x(1) = \frac{(1-\gamma)x_1+x_2}{2-\gamma}. 
\end{equation}
In other words, Euler summation assigns to the (potentially divergent) series the value of its analytic continuation at $t=1$. This provides a well-defined finite value for a series defined by a recurrence relation, even if the original series is divergent.

\begin{lemma}
	\label{lem:sum}
	Suppose that the sequence $(x_n)_{n=1}^{\infty}$ fulfils the recurrence relation
$$x_{n+2} = (r-4)x_{n+1} - x_n,$$
where $r \neq 6$. Then the Euler sum of its elements is given by:
$$(r-6 )\sum_{n=1}^{\infty} x_n = (r-5)x_1 - x_2.$$
\end{lemma}
\begin{proof}
	Substituting $\gamma=r-4$ into (\ref{def:euler_sum}), we obtain
\begin{align*}
	(r-6) \sum_{n=1}^{\infty} x_n &= (r-6) \frac{ (\gamma-1)x_1 - x_2}{\gamma -2} \\
	&= (r-5) x_1 - x_2.
\end{align*} 
	
\end{proof}

\subsection{Graph}

By a \emph{graph} we mean an undirected simple graph, i.e., a graph whose 
edges are unoriented and which contains neither loops nor multiple edges. 
A graph may be either finite or infinite. 

The \emph{degree} of a vertex $w$ in a graph $G$ is 
the number of edges of $G$ incident to $w$, denoted by $\deg_G(w)$.
A \emph{path} 
between vertices $w_1$ and $w_t$ in a graph $G$ is a sequence of vertices
$(w_1, w_2, \ldots, w_t)$ such that consecutive vertices are adjacent in $G$.
A graph is called connected if there exists a path between any two of its vertices.
A \emph{cycle} is a path $(w_1, w_2, \ldots, w_t)$ of length $t \geq 3$ in which the first and last vertices $w_1$ and $w_t$ are adjacent. 
A simple path or simple cycle is a path or cycle in which all vertices are distinct.
Formally, we identify a path or a cycle with the subgraph formed by its vertices and edges.

A graph is \emph{planar} if it can be drawn in the plane so that its edges do not cross. 
A \emph{plane graph} is a planar graph together with a fixed embedding in the Euclidean plane.
Formally, a plane graph $G$ is a triple $G = (V, E, \varphi)$, where $V$ is a set of vertices, 
$E$ is a set of edges, and $\varphi : G \rightarrow \mathbb{R}^2$ is an embedding in the Euclidean plane. For convenience, we simply write $V_G=V$ and $E_G=E$.
In this paper, all graphs are considered up to combinatorial equivalence of their embeddings in the plane.
Therefore, we will refer to plane graphs simply as graphs and write $G=(V,E)$, 
assuming that their planar embedding is fixed.

A \emph{face} of a graph drawn in the plane is a connected region of the plane enclosed 
by its edges, whose interior contains no vertices or edges. For finite graphs, one of these regions 
extends to infinity (more precisely, it is the complement in the plane of the region 
occupied by the drawing of the graph), and this region is called the \emph{outer face}.

\subsection{Euler’s formula for finite planar graph}
Let $G$ be a finite connected planar graph drawn in the plane, and let $v, e, f$ denote respectively 
the numbers of its vertices, edges, and interior faces. A fundamental property of such graphs is expressed 
by the Euler formula, which states that
$$v-e+f=1,$$
where here, $f$ counts only the finite faces (excluding the outer face).
In this paper, we will define analogues of the numbers of vertices, edges, and faces for a mesh. 
As we will show, the quantities we will define will also satisfy Euler's formula.

\subsection{Mesh}

A \emph{mesh} is a connected infinite planar graph drawn in the plane without edge crossings, 
that satisfies the following conditions.
\begin{itemize}
	\item Each vertex has the same finite degree $r$. This number $r$ is called 
	the \emph{degree of the mesh}.
	\item Each face is a triangle, i.e. a region bounded by three distinct edges.
	\item Every bounded region of the plane contains only finitely many vertices and edges.
\end{itemize}
\begin{remark}
Such a graph can be drawn in the hyperbolic plane so that 
each of its faces is an equilateral triangle. The graph drawn in the hyperbolic plane can be projected 
onto the Euclidean plane, since it is connected, planar, and locally finite.
\end{remark}

\subsection{Combinatorial disk and convexity}

Let $M$ be the mesh of degree $r>6$ and let $C \subset M$ be a simple cycle of the mesh $M$. 
A graph $G$ consisting of the cycle $C$ and all vertices and edges inside 
the cycle $C$ will be called a \emph{combinatorial disk}. The cycle $C$ is called the boundary of $G$. 
In other words, a combinatorial disk is a graph occupying a connected region of the plane. 
A combinatorial disk is a combinatorial analogue of a polygonal disk in a plane.  

A combinatorial disk $G$ with boundary $C$ is called \emph{convex} if and only if 
every vertex of the cycle $C$ has a degree in the graph $G$ 
that does not exceed $2+\frac{r-2}{2}$, i.e.
$$\underset{w \in V_C}{\forall} \deg_G(w) \leq 1+\frac{r}{2}.$$
In other words, $G$ is convex if at each vertex, among $r-2$ edges of the mesh that do not belong to the cycle $C$, at most half belong to the graph.
The convexity of a combinatorial disk is a combinatorial analogue of the convexity of a polygon in a plane.  

\begin{lemma}
	\label{lem:initial}
Let $G \subset M$ be a convex combinatorial disk with the boundary $C~=~(w_1, \dots, w_t)$, and let $v, e, f$ denote respectively the number of its vertices, edges, and (interior) faces. We define the sum of degrees of the boundary $d~=~\sum\limits_{j=1}^{t} \deg_{G}(w_j)$.
Then 
\begin{equation*}
	\left\{
	\begin{aligned}
		(r-6) v &= tr - 2t -d -6 \\
		(r-6) e &= 2tr -3d -3r \\
		(r-6) f &= tr+2t-2d-2r
	\end{aligned}
	\right.
\end{equation*}
\end{lemma}
\begin{proof}
	We use four relations concerning the graph $G$. 
	First, Euler's formula gives $v - e + f = 1$.  
	Second, let $s=v-t$ denote the number of interior vertices of $G$. (Each of them has degree $r$.) 
	Third, since each edge has two endpoints, the sum of all vertex degrees $rs+d$ is equal to $2e$.  
	Finally, each face is triangular and the outer face has $t$ sides; as each edge is incident to two faces, we obtain $3f + t = 2e$.  
	Merging these relations, we obtain the following system of equations:
	\begin{equation*}
		\left\{
		\begin{aligned}
			v+f  & =  e+1,\\
			t+s  & =  v,\\
			rs+d & =  2e,\\
			3f+t & =  2e.
		\end{aligned}
		\right.
	\end{equation*}
	We treat the variables $t, d$, and $r$ as parameters 
	and solve the system of equations using the standard Gaussian elimination. Since $r-6 \neq0$, the system has unique solution given by:
	\begin{equation*}
		\left\{
		\begin{aligned}
			(r-6) v &= tr - 2t -d -6, \\
			(r-6) e &= 2tr -3d -3r, \\
			(r-6) f &= tr+2t-2d-2r, \\
			(r-6) s &= 4t-d-6.
		\end{aligned}
		\right.
	\end{equation*}
	We omit the Gaussian elimination, as the solution can be easily verified.
\end{proof}

\subsection{Graph expansion}
Let $M$ be a mesh of degree $r$. For any non-empty subgraph $G \subset M$, we define the expansion of $G$ as \( T(G)=H \subset M \), where 
$$
\begin{array}{@{}r@{ }l@{}}
V_{H} &= V_G \cup \left\{ w \in \big(V_M \setminus V_G\big) \middle| \underset{u \in V_G}{\exists} (u, w) \in E_M \right\} \\
E_H &= \big\{ \left(w_1, w_2\right) \in E_M | w_1, w_2 \in V_H \big\} \\
\end{array}
$$
In other words, $T(G)$ is an extension of $G$ obtained by adding to it all vertices and edges of faces that share at least one vertex with $G$. 

\begin{lemma}
	\label{lem:vfe}
	Let $G \subset M$ be a convex combinatorial disk with the boundary $C=(w_1, \dots, w_t)$ and $H=T(G)$ be its expansion. 
	Then $H$ is convex and newly added vertices form a simple cycle $D$ which is boundary of $H$, and each of them $u \in V_D$ has degree $\deg_H(u) \in \{3, 4\}$.    
	
	\begin{enumerate}[(a)]
		\item \label{lem:a}  
	Let $a$ and $b$ be the numbers of vertices of $H$ with degree $3$ or $4$, respectively.
	Then
	\begin{equation*}
		\left\{
		\begin{aligned}
			a &= t(r-2)-\sum\limits_{j=1}^{t} \deg_G(w_j) , \\
			b &= t
		\end{aligned}
		\right.
	\end{equation*}
	\item \label{lem:b}
	Let $v, e, f$ denote the numbers of newly added vertices, faces, and edges, i.e. 
	those that belong to $H$ but not to $G$. Then
	\begin{equation*}
		\left\{
		\begin{aligned}
			v &= a+b,\\
			e &= 2a+3b, \\
			f &= a+2b,
		\end{aligned}
		\right.
	\end{equation*}
	\end{enumerate}
\end{lemma}
\begin{proof}
	Newly added vertices $V_{H} \setminus V_{G}$ 
	form a simple cycle $D$, since $G$ is convex.
			
	Each vertex of the mesh $M$ has degree $r$, then each vertex $w \in V_{C}$ 
	is adjacent with $deg_{G}(w)$ vertices of $G$  (including predecessor and successor in the cycle $C$), 
	and thus with $r-deg_{G}(w)$ vertices of $H$. 
	
	Let $w_1, w_2$ be any two consecutive vertices of cycle $C$. 
	Recall that each face is a triangle, and each edge belongs to two faces. 
	Therefore, there exist exactly two vertices $u$ such that
	$w_1, w_2, u$ are the vertices of a single face. 
	One of them is located outside the cycle $C$, and the other inside the cycle $C$. 
	Considering the cyclic order of edges incident to a vertex $w \in C$, one observes the following pattern:
	\begin{itemize}
		\item the edge adjacent to the vertex $u \in V_{D}$ with degree $4$ in $H$ which is common with the predecessor on $C$,
		\item $deg_{G}(w)$ edges adjacent with vertices in $G$
		\item the edge adjacent to the vertex $u' \in V_{H}$ with degree $4$ in $H$ which is common with the successor on $C$,
		\item $r-2-deg_G(w)$ edges adjacent to the vertices with degree $3$ in $H$.
	\end{itemize}
	Therefore any vertex $u \in H$ has degree $\deg_{H}(u) \in \{3,4\}$. Since
	$$1+\frac{r}{2} > 2+\frac{6}{2} = 4,$$
	then the convexity condition on vertex degrees is satisfied and the graph $H$ is convex. 
	Furthermore, the number of vertices with degree $4$ is equal to the length of the cycle $C$, i.e.
	$$b =  t$$
	and the numbers of vertices with degree $3$ is equal to the number of other edges
	$$ a = \sum\limits_{ w \in V_{C}} r-2-deg_G(w) = (r-2)t - \sum\limits_{j=1}^{t} \deg_G(w_j).$$

	Furthermore, using simple observations, we will count the number of newly added vertices, edges, and faces. 
	First, each newly added vertex $u \in D$ has degree $\deg_H(u) \in \{3, 4 \}$ 
	including one edge adjacent to the predecessor on $C$ and one edge adjacent to the successor on $C$. 
	Second, each vertex $u \in D$ is incident to $\deg_H(u)-1$ faces, including one common face with predecessor on $D$ and one common face with successor on $D$.
	Thus:
	\begin{equation*}
		\left\{
		\begin{aligned}
			v &= a+b,\\
			e &= 3a+4b - v = 2a+3b, \\ 
			f &= (3-1)a+(4-3)b - v = a+2b.
		\end{aligned}
		\right.
	\end{equation*}
\end{proof}

\section{Counting vertices, edges, and faces in iterated graph expansions}

In this section we provide recursive formulas for the number of vertices, edges, and faces 
in graphs obtained by iterative repeated graph expansion.

Let $M$ be the mesh of degree $r > 6$ and let $G_0 \subset M$ be a convex combinatorial disk.
Define the sequence of graphs $(G_n)_{n=0}^{\infty}$ by the recurrence
$G_{n+1}=T(G_n)$ for all $n \in \mathbb{N}$. 
In other words $G_n = T^n(G_0)$ is a graph obtained by applying expansion $n$ times starting from the initial graph $G_0$.
A sample visualization of the first few expansions of the single-vertex graph is shown in \cref{fig:vertex_expansion}.
\mesh

We denote the boundary of $G_0$ by $C=(w_1, \dots, w_t)$ and the sum of degrees of boundary by $d=\sum\limits_{j=1}^{t} \deg_{G_0} (w_j)$. 
Let $a_n$ and $b_n$ for $n \ge 1$, denote the number of vertices of $G_n$ with degree $3$ and $4$ in $G_n$, respectively. Let $v_n, e_n$, and $f_n$ denote the number of newly added vertices, edges, and faces in $n$-th expansion $G_n \rightarrow G_{n+1}$.

We now apply \cref{lem:vfe} to derive relationships between the quantities just introduced.
First, we note that simple induction provides that $G_{n-1}$ is convex for any $n \in \mathbb{N}_+$ and thus the boundary of $G_{n}$ consists of only vertices with degree $3$ or $4$ in $G_{n}$, and remaining vertices of $G_{n}$ have degree equal to $r$ in $G_{n}$. 
Convexity is preserved under expansion, since $1+\frac{r}{2} > 1+\frac{6}{2} = 4$.
 
Applying part (\ref{lem:a}) of \cref{lem:vfe} for $G=G_n$ we obtain some result at the initial step $n=0$ and different results for the remaining steps.
\begin{itemize}
	\item For $n=0$ we obtain initial values: 
	\begin{equation}
		\label{eq:initial}
	\left\{
	\begin{aligned}
		a_1 &= t(r-2)-d \\
		b_1 &= t
	\end{aligned}
	\right.
	\end{equation}
	\item For $n \geq 1$ we obtain a coupled recurrence relation:
	\begin{equation}
		\label{eq:reccurssion}
	\left\{
	\begin{aligned}
		a_{n+1} &= (a_n+b_n)(r-2)-(3a_n+4b_n) \\&= (r-5)a_n + (r-6)b_n, \\
		b_{n+1} &= a_n+b_n,
	\end{aligned}
	\right.
	\end{equation}
	since the boundary of $G_n$ has $a_n+b_n$ vertices. 
\end{itemize} 
 
Part \cref{lem:b} provides relationship with remaining quantities:
\begin{equation}
	\label{eq:linear}
	\left\{
	\begin{aligned}
		v_n &= a_n+b_n,\\
		e_n &= 2a_n+3b_n, \\
		f_n &= a_n+2b_n.
	\end{aligned}
	\right.
\end{equation}

For the remaining quantities, we obtain the following simple recurrence.
\begin{proposition}
	\label{prop:recall}
Let $(x_n)_{n=0}^{\infty}$ be a fixed among three sequences $(v_n)_{n=1}^{\infty}, (e_n)_{n=1}^{\infty},$ and $(f_n)_{n=1}^{\infty}$.
It satisfies the following recurrence
\begin{equation}
	\label{rec:main}
	x_{n+2} = (r-4) x_{n+1} - x_{n}
\end{equation}
for all $n \in \mathbb{N}_+$.
\end{proposition}

\begin{proof}
	First, we will prove that the recurrence is satisfied by the sequences $(a_n)_{n=1}^{\infty}$ and $(b_n)_{n=1}^{\infty}$. 
	Using \cref{eq:reccurssion} we obtain that for any $n \in \mathbb{N}_+$ holds:
	\begin{align*}
		\qquad\qquad a_{n+2} - a_{n+1} 
		&= (r-5) (a_{n+1} - a_n) + (r-6) (b_{n+1}-b_n) \\
		&= (r-5) (a_{n+1} - a_n) + (r-6) a_n \\
		&= (r-5)a_{n+1} - a_n \\
	\end{align*}
	and 
	\begin{align*}
		b_{n+2} - (r-5) b_{n+1}
		&= (a_{n+1}+b_{n+1})- (r-5)(a_n + b_n) \\
		&= a_{n+1}-(r-5) a_n + b_{n+1}-(r-5) b_n \\
		&= (r-6) b_n + b_{n+1}-(r-5) b_n \\
		&= b_{n+1}-b_n.
	\end{align*}
	Thus, the sequences $(a_n)_{n=1}^{\infty}$ and $(b_n)_{n=1}^{\infty}$ satisfy the recurrence \cref{rec:main}, i.e.:
	\begin{equation}\label{eq:pair}
		\left\{
		\begin{aligned}
			a_{n+2} = (r-4) a_{n+1} - a_{n}, \\
			b_{n+2} = (r-4) b_{n+1} - b_{n}
		\end{aligned}
		\right.
	\end{equation}
	for any $n \in \mathbb{N}_+$.

	Let $(x_n)_{n=0}^{\infty}$ be a fixed sequence among $(v_n)_{n=1}^{\infty}$, $(e_n)_{n=1}^{\infty}$, or $(f_n)_{n=1}^{\infty}$. 
	Observe that by \eqref{eq:linear}, the element $x_n$ can be written as
	$x_n=\alpha a_n+\beta b_n$ for fixed constants $\alpha, \beta$.
	Thereby 
	\begin{align*}
		&x_{n+2}-(r-4)x_{n+1}+x_n  \\ 
		&= (\alpha a_{n+2} + \beta b_{n+2}) - (r-4)(\alpha a_{n+1} + \beta b_{n+1}) + (\alpha a_n + \beta b_n) \\
		&= \alpha (a_{n+2} - (r-4) a_{n+1} + a_n)+ \beta (b_{n+2} - (r-4) b_{n+1} + b_n) \\
		&= 0
	\end{align*} 
	by \cref{eq:pair}. 
\end{proof}

Recall that $M$ is the mesh of degree $r > 6$, initial graph $G_0 \subset M$ is a convex combinatorial disk, and the sequence $(G_n)_{n=0}^{\infty}$ is given by the recurrence
$G_{n+1}=T(G_n)$ for all $n \in \mathbb{N}$. For the mesh $M$, we define the number of its vertices $v_M = \sum\limits_{n=0}^{\infty} v_j$, its edges $e_M = \sum\limits_{n=0}^{\infty} e_j$, and its faces $f_M = \sum\limits_{n=0}^{\infty} f_j$ by the Euler sum. 
We will show that this definition is independent of the choice of the initial disk
Since the sequences $(v_n)_{n=0}^{\infty}$, $(e_n)_{n=0}^{\infty}$, $(f_n)_{n=0}^{\infty}$ satisfy the recurrence \cref{rec:main}, their Euler sums are well-defined.
\begin{theorem}
	\label{thm:main}
The number of vertices, edges, and faces of the mesh $M$ depends solely on de degree of the mesh and are independent of the choice of the initial convex combinatorial disk $G_0$. 
More precisely:
\begin{equation*}
	\left\{
	\begin{aligned}
		v_M &= \frac{-6}{r-6}, \\
		e_M &= \frac{-3r}{r-6}, \\
		f_M &= \frac{-2r}{r-6}.
	\end{aligned}
	\right.
\end{equation*}

\begin{remark}
	These quantities should be viewed as Euler-summed invariants; they are not 
	cardinalities of infinite sets and can take non-integer or even negative values.
\end{remark}
\end{theorem}

\begin{proof}
Since the sequences $(v_n)_{n=1}^{\infty}, (e_n)_{n=1}^{\infty},$ and $(f_n)_{n=1}^{\infty}$ satisfy the recurrence relation \cref{rec:main} by \cref{prop:recall}, 
then we applying \cref{lem:sum} and from \label{eq:vef} we obtain
\begin{equation*}
	\left\{
	\begin{aligned}
		(r-6) (v_M - v_0) 
		&=  (r-5)v_1 - v_2, \\
		&=  (r-5)(a_1+b_1) - a_2-b_2 \\
		(r-6) (e_M  - e_0)
		&= (r-5)e_1 - e_2, \\
		&=  (r-5)(2a_1+3b_1) - 2a_2-3b_2\\
		(r-6) (f_M - f_0)
		&= (r-5)f_1 - f_2 \\
		&=  (r-5)(a_1+2b_1) - a_2-2b_2
	\end{aligned}
	\right.
\end{equation*}
Next, using the recurrence \eqref{eq:reccurssion} together with the initial values \eqref{eq:initial}, we obtain
\begin{equation*}
	\left\{
	\begin{aligned}
		(r-6) (v_M - v_0)
		&= (r-5)(a_1+b_1) - \big((r-5)a_1+(r-6)b_1\big)-(a_1+b_1) \\
		&= -a_1 \\
		&= -(tr-2t-d) \\
		(r-6) (e_M - e_0),
		&= (r-5)(2a_1+3b_1) -2\big((r-5)a_1+(r-6)b_1\big)-3(a_1+b_1)\\
		&= -3a_1+(r-6)b_1 \\
		&= -3(tr-2t-d)+(r-6)t \\
		(r-6) (f_M - f_0),
		&=  (r-5)(a_1+2b_1) -\big((r-5)a_1+(r-6)b_1\big)-2(a_1+b_1) \\
		&= -2a_1+(r-6)b_1 \\
		&= -2(tr-2t-d)+(r-6)t.
	\end{aligned}
	\right.
\end{equation*}
Finally, using \cref{lem:initial} for $G=G_0$ we obtain formulas related only with degree of the mesh:
\begin{equation*}
	\left\{
	\begin{aligned}
		(r-6) v_M 
		&= (r-6) v_0 -(tr-2t-d) \\
		&= tr - 2t -d -6 -tr+2t+d \\
		&= -6, \\ 
		(r-6) e_M 
		&= (r-6)e_0 -3(tr-2t-d)+(r-6)t \\
		&= 2tr -3d -3r -3tr+6t+3d+rt-6t  \\
		&= -3r, \\
		(r-6) f_M 
		&=  (r-6) f_0 -2(tr-2t-d)+(r-6)t \\
		&=  tr+2t-2d -2tr+4t+2d+rt-6t \\
		&= -2r.
	\end{aligned}
	\right.
\end{equation*}
By dividing both sides of each equation by $r-6$, we obtain the thesis. 
\end{proof}

\begin{corollary}
	The numbers $v_M$, $e_M$, $f_M$ satisfy Euler's formula:
	$$v_M - e_M + f_M = 1.$$
\end{corollary}
\begin{proof}
	From \cref{thm:main} we obtain
	$$v_M - e_M + f_M = -\frac{6}{r-6} + \frac{3r}{r-6} - \frac{2r}{r-6} = \frac{r-6}{r-6}  = 1.$$
\end{proof}

\bibliographystyle{plain}
\bibliography{reference}

\end{document}